\documentclass[amstex,10pt,reqno]{amsart}
\usepackage{amsmath,amsfonts,amssymb,amsthm,enumerate,hyperref,multicol, graphicx, pdfsync,array, caption, subcaption, color}
\usepackage[dvipsnames]{xcolor}
\newtheorem{lemma}{Lemma}
\newtheorem{thm}{Theorem}
\newtheorem{definition}{Definition}
\newtheorem{example}{Example}

\numberwithin{equation}{section}
\begin{document}
\leftline{ \scriptsize \it}
\title[]
{Approximation of function by $\alpha-$Baskakov Durrmeyer type operators}
\maketitle
\begin{center}
{\bf  Jaspreet Kaur$^1$ and Meenu Goyal$^1$}
\vskip0.2in
$^{1}$School of Mathematics\\
Thapar Institute of Engineering and Technology, Patiala\\
Patiala-147004, India\\
\vskip0.2in
jazzbagri3@gmail.com and meenu\_rani@thapar.edu
\end{center}
\begin{abstract}
In the present note, we give the generalization of $\alpha-$Baskakov Durrmeyer operators depending on a real parameter $\rho>0.$ We present the approximation results in Korovkin and weighted Korovkin spaces. We also prove the order of approximation, rate of approximation for these operators. In the end, we verify our results with the help of numerical examples by using Mathematica.\\


\end{abstract}
\section{\bf{{Introduction}}}
Weierstrass's approximation theorem is the stepping stone of approximation theory, proved by Weierstrass in 1885. In 1912, S. N. Bernstein presented an another method to prove the theorem, which is quite simpler than the original proof. For this, he defined positive and linear operators (named after him) as Bernstein operators. For $f\in C[0,1],$ the operators  are defined as:
$$\mathcal{B}_n(f,x)=\sum_{k=0}^np_{n,k}(x)f\left(\frac{k}{n}\right), $$
where $p_{n,k}(x)= \displaystyle {n \choose k}  x^k  (1-x)^{n-k}.$ These operators enjoy remarkable properties of approximation theory. Due to its useful properties,  immense research has been initiated on these operators, which is still going on (one can see \cite{BER},\cite{PLB},\cite{DT},\cite{CAM}). We can also find its generalizations based on different parameters (see \cite{lambe},\cite{CHE},\cite{two},\cite{QBER},\cite{stan}). Despite having all these interesting properties, the operators are defined only on [0,1]. Then in 1950, Sz$\acute{a}$sz \cite{szs} extended it to the infinite interval by introducing Sz$\acute{a}$sz operators.\\
In 1957, Baskakov \cite{BAS} acquainted a sequence of positive linear operators for $f\in C_B[0,\infty)$ i.e. continuous bounded functions on $[0,\infty)$ as:
\begin{eqnarray}
L_n(f;x)=\sum_{k=0}^{\infty} f\left(\frac{k}{n}\right)l_{n,k}(x),     \qquad     n \geq 1,\qquad x \in [0,\infty)
\end{eqnarray}
where $l_{n,k}(x)=\displaystyle {n+k-1 \choose k}\displaystyle \frac{x^k}{(1+x)^{n+k}}$.\\
 Due to the extension of these operators on unbounded interval, it fetched the attention of many researchers \cite{Baka}, \cite{kanb}, \cite{Bask}, \cite{Ba}, \cite{bas}.\\
Recently, Aral and Erbay \cite{albas} introduced the generalization of classical Baskakov operators based on a real parameter $\alpha\in [0,1]$, as :
\begin{eqnarray}
L_{n}^{\alpha}(f;x)= \sum_{k=0}^{\infty} f\left(\frac{k}{n}\right).p_{n,k}^{\alpha}(x), \qquad        n \geq 1, \qquad x \in [0,\infty)
\end{eqnarray}
where $p_{n,k}^{\alpha}(x)$ is given by:
\begin{eqnarray}
p_{n,k}^{\alpha}(x)=\frac{x^{k-1}}{(1+x)^{n+k-1}} \left\{\frac{{\alpha}x}{1+x}{n+k-1 \choose k}-(1-\alpha)(1+x){n+k-3 \choose k-2}+(1-\alpha)x{n+k-1 \choose k}\right\}.
\end{eqnarray}
For $\alpha=1,$ it comes back to the classical Baskakov operators.\\
The authors studied convergence, error approximation and Voronovskaja type result for these operators. They verified that the parameter $\alpha$ does not affect the convergence but it effects the error of approximation for these operators.\\
 Very recently, Nasiruzzaman et al. \cite{dalba} defined the Durrmeyer variant of these operators to extend the results on Lebesgue integral functions, called $\alpha-$Baskakov Durrmeyer operators as:
\begin{eqnarray}\label{op}
S_{n}^{\alpha}(f;x)= \sum_{k=0}^{\infty} f\left(\frac{k}{n}\right).p_{n,k}^{\alpha}(x)\int_{0}^{\infty}s_{n,k}(t)f(t)\, dt,\\
\mbox{where}\quad s_{n,k}(t)=\frac{t^k}{B(k+1,n)(1+t)^{n+k+1}}\nonumber
\end{eqnarray}
and $B(m,n)$ is a Beta function.
The authors studied the order of approximation, rate of convergence, Korovkin-type and weighted Korovkin-type approximation theorems for these operators.\\
In order to reduce the error of approximation of the linear positive operators, a great variety of research is going on its generalization which is based on different parameters (as \cite{albas}, \cite{qbas}, \cite{genqbas}, \cite{geba}, \cite{wafi}). Although, generalization does not give better approximation in all the cases. But, if we choose the specific value of the parameter wisely for the certain function, we can achieve the great results.
 By motivated from these benefits of the generalizations of the positive linear operators, which reduces the error of approximation, we modify the operators (\ref{op}) by using a parameter $\rho>0$ in the next section.
\section{Construction of operators}
We define the generalization of $\alpha-$Baskakov Durrmeyer operators based on a real parameter $\rho >0$ as follows:
\begin{eqnarray}\label{ope}
A_{n}^{\alpha,\rho}(f;x)&=&\sum_{k=0}^{\infty} f\left(\frac{k}{n}\right).p_{n,k}^{\alpha}(x)\int_{0}^{\infty}\mu_{n,k}^{\rho}(t)f(t) \,dt,\\
\mbox{where} \quad
\mu_{n,k}^{\rho}(t)&=&\frac{t^{k\rho}}{B(k\rho+1,n\rho)(1+t)^{n\rho+k\rho+1}}.
\end{eqnarray}
The following sections are organized as: In section 3, we give the moments, central moments and some preliminary results that we will use in subsequent sections. In section 4, we show main results for the operators (\ref{ope}) which are divided in two sub-sections. In the first sub-section, we prove basic convergence theorems in Korovkin and weighted Korovkin spaces and in the second sub-section, we study pointwise approximation properties of these operators. In section 5, we verify our theoretical results by numerical examples with the use of Mathematica.\\
Through out the paper, we use $e_i=x^i,$ for $i=0,1,2,\dotsb .$
\section{Basic Results}
\begin{lemma}\label{mome}
The moments of the operators $A_{n}^{\alpha,\rho}(.;x)$ are given as:
\begin{eqnarray*}
A_{n}^{\alpha,\rho}(e_0;x) &=& 1;\\
A_{n}^{\alpha,\rho}(e_1;x) &=& \frac{\rho nx-2\rho(1-\alpha)x+1}{n\rho-1};\\
A_{n}^{\alpha,\rho}(e_2;x) &=& \frac{1}{(n\rho-1)(n\rho-2)}\bigg[x^2\big\{{\rho}^2 n^2+(4\alpha-3){\rho}^2n\big\}\\
&&+x\big\{{\rho}^2(n+4\alpha-4)+3\rho n-6\rho(1-\alpha)\big\}+2\bigg].
\end{eqnarray*}
\end{lemma}
\begin{lemma}\label{cenm}
The central moments of the operators $A_{n}^{\alpha,\rho}(.;x)$ are given as:
\begin{eqnarray*}
A_{n}^{\alpha,\rho}(t-x;x) &=& \frac{1}{n\rho-1}\bigg[x\left\{1-2\rho(1-\alpha)\right\}+1\bigg];\\
A_{n}^{\alpha,\rho}((t-x)^2;x) &=& \frac{1}{(n\rho-1)(n\rho-2)}\bigg[x^2\big\{n\rho(\rho+1)-8\rho(1-\alpha)+2\big\}\\
&&+x\big\{n\rho(\rho+1)-4{\rho}^2(1-\alpha)-6\rho(1-\alpha)+4\big\}+2\bigg];\\
A_{n}^{\alpha,\rho}((t-x)^4;x)&=& \frac{{\rho}^2(1+\rho)^2x^2(1+x)^2-96(1-\alpha){\rho}^3x^3+24{\rho}^3x^3}{(n\rho-1)(n\rho-2)(n\rho-3)(n\rho-4)}n^2+O\left(\frac{1}{n^3}\right).
\end{eqnarray*}
\end{lemma}
\vskip0.1in
Now, we recall the definitions and the results that are useful in the subsequent sections.

\begin{definition}{\bf{(First order modulus of continuity)}}\label{def1}
Let $f(x)$ be bounded on $[a,b],$ the modulus of continuity of $f(x)$ on $[a, b],$ is denoted by $\omega(f;\delta),$ and is defined for $\delta>0$ as:
\begin{eqnarray*}
\omega(f;\delta)=\sup_{\substack{(x,y)\in [a,b] \\ |x-y|\leq \delta}} |f(x)-f(y)|,
\end{eqnarray*}
with the following relation for $\lambda>0$
\begin{eqnarray}\label{a8}
\omega(f;\lambda \delta)\leq (1+\lambda)\,\omega(f;\delta).
\end{eqnarray}
\end{definition}
\vskip0.1in
\begin{definition}{\bf{(Lipschitz continuity)}}\label{def2}
A function $f$ satisfies the Lipschitz condition of order $\tau$ with the constant $L$ on $[a,b],$ i.e.
\begin{eqnarray*}
|f(x_1)-f(x_2)|\leq L\, |x_1-x_2|^\tau,\qquad 0< \tau\leq 1,\quad L>0, \quad \mbox{for}\quad x_1,x_2\in [a,b],
\end{eqnarray*}
iff $\omega(f;\delta)\leq L\, \delta^\tau.$
\end{definition}
\vskip0.1in
\begin{definition}{\bf{($K-$Functional)}}\label{def3}
Peetre \cite{kfu} introduced the second order $K-$functional for $f\in C_B[0,\infty)$:
\begin{eqnarray*}
K_2(f;\delta)=\displaystyle \inf\{\|f-g\|_{C_B[0,\infty)}+\delta \|g^{\prime\prime}\|_{C_B^2[0,\infty)}; g\in C_B^2[0,\infty)\},
\end{eqnarray*}
where $$C_B^2[0,\infty):=\{f\in C_B[0,\infty);  f^{\prime},f^{\prime\prime}\in C_B[0,\infty)\}.$$
\end{definition}
\vskip0.1in
In order to prove the uniform convergence of the positive linear operators to a continuous function, Korovkin has proved a simple theorem which is defined as:\\
\begin{thm}\label{thm}
Let $(L_n)_{n\geq 1}$ be a sequence of positive linear operators such that for every $f\in \{e_0,e_1,e_2\}$
$$\lim_{n\rightarrow \infty}L_n(f)=f \qquad \mbox{uniformly on}\qquad [a,b].$$
Then for a function $h\in C[a,b],$ we have
$$\lim_{n\rightarrow \infty}L_n(h)=h \qquad \mbox{uniformly on}\qquad [a,b].$$
\end{thm}
\vskip0.1in
To get the uniform convergence on an unbounded interval, we define the following conditions:\\
Let $f\in C_B^2[0,\infty)$ with
\begin{eqnarray}\label{norm}
\|f\|_{C_B^2[0,\infty)}&=& \|f\|_{C_B[0,\infty)}+\|f^{\prime}\|_{C_B[0,\infty)}+\|f^{\prime\prime}\|_{C_B[0,\infty)}\\
\mbox{where} \qquad\|f\|_{C_B[0,\infty)}&=&\displaystyle \sup_{x \in [0,\infty)} |f(x)|.\nonumber\\
\mbox{Define}\qquad D[0,\infty) &:=& \{f\in C_B[0,\infty):x\in[0,\infty) \lim_{x \rightarrow \infty}\frac{f(x)}{1+x^2}<\infty\}.\nonumber
\end{eqnarray}
Suppose $D_2[0,\infty):=\{f:|f(x)| \leq A_f(1+x^2)\}$
such that $A_f >0,$ a constant number depends only on $f.$
\section{Main results}
\subsection{Approximation in Korovkin and weighted Korovkin spaces}
\vskip0.2in
\begin{thm}
For every $f\in C[0,\infty) \cap D[0,\infty),$ the operators (\ref{ope}) are uniformly convergent to $f$ on each compact subset of $[0,A],$ where $A \in (0,\infty)$.
\end{thm}
\begin{proof}
For the uniform convergence of the operators (\ref{ope}) to $f\in C[0,\infty) \cap D[0,\infty),$ we need to prove:\\
 $A_{n}^{\alpha,\rho}(e_i;x)\rightarrow e_i,$  for $i=0,1,2.$\\
From Lemma \ref{mome}, it is obvious that $A_{n}^{\alpha,\rho}(e_0;x)\rightarrow e_0,$ as $n\rightarrow \infty.$\\
For $i=1,$ we can check from below:
\begin{eqnarray*}
\displaystyle\lim_{n\rightarrow \infty} A_{n}^{\alpha,\rho}(e_1;x)&=&\lim_{n\rightarrow \infty}\frac{n\rho x-2\rho(1-\alpha)x+1}{n\rho-1}\\
&=&\lim_{n\rightarrow \infty}\left(x+\frac{x-2\rho(1-\alpha)x+1}{n\rho-1}\right)=x.\\
\end{eqnarray*}
Now, for $i=2$
\begin{eqnarray*}
\lim_{n\rightarrow \infty}A_{n}^{\alpha,\rho}(e_2;x)&=&\lim_{n\rightarrow \infty}\frac{1}{(n\rho-1)(n\rho-2)}\bigg[x^2\big\{{\rho}^2n^2+(4\alpha-3){\rho}^2n\big\}\\
&&+x\big\{{\rho}^2(n+4\alpha-4)+3\rho n-6\rho(1-\alpha)\big\}+2\bigg]\\
&=& x^2+ \lim_{n\rightarrow \infty}\frac{1}{(n\rho-1)(n\rho-2)}\bigg[x^2\big\{3\rho n-2+(4\alpha-3){\rho}^2n\big\}\\
&&+x\big\{{\rho}^2(n+4\alpha-4)+3\rho n-6\rho(1-\alpha)\big\}+2\bigg]=x^2.\\
\end{eqnarray*}
Hence, by using Theorem \ref{thm}, we get the uniform convergence of our operators.
\end{proof}

For the approximation in weighted Korovkin spaces, we define the spaces as:
$$C_2[0,\infty)=D_2[0,\infty)\cap C[0,\infty)$$
and
$C_2^*[0,\infty)=\left\{f \in C_2[0,\infty):\displaystyle \lim_{x \rightarrow \infty}\frac{|f(x)|}{1+x^2}<\infty \right\}$ with $\|f\|_2=\displaystyle \sup_{x\in[0,\infty)}\frac{|f(x)|}{1+x^2}.$
\begin{thm}
For $f\in C_2^*[0,\infty),$ the operators $A_{n}^{\alpha,\rho}(f;x)$ satisfy:
$$\displaystyle\lim_{n\rightarrow \infty} \|A_{n}^{\alpha,\rho}(f;.)-f\|_2=0.$$
\end{thm}

\begin{proof}
By weighted Korovkin thm, it is sufficient to prove:
$$\displaystyle\lim_{n\rightarrow \infty}\|A_{n}^{\alpha,\rho}(e_i;.)-e_i\|_2=0,\quad \mbox{  for} \quad i=0,1,2.$$
For $i=0$, using Lemma \ref{mome}, we have $A_{n}^{\alpha,\rho}(e_0;x)=1.$\\
Therefore, we have $\|A_{n}^{\alpha,\rho}(e_0;.)-e_0\|_2=0.$ \\
For $i=1,$
\begin{eqnarray*}
\|A_{n}^{\alpha,\rho}(e_1;x)-e_1\|_2&=&\displaystyle \sup_{x\in[0,\infty)}\frac{1}{1+x^2}\left|\frac{n\rho x-2\rho(1-\alpha)x+1}{n\rho-1}-x\right|\\
&\leq & \displaystyle \sup_{x\in[0,\infty)}\frac{x}{1+x^2}\left|\frac{1-2\rho(1-\alpha)}{n\rho-1}\right|+\displaystyle \sup_{x\in[0,\infty)}\frac{1}{1+x^2}\left|\frac{1}{n\rho-1}\right|.
\end{eqnarray*}
Thus, we get:
\begin{eqnarray*}
\|A_{n}^{\alpha,\rho}(e_1;x)-e_1\|_2=0 \,\mbox{as}\,\, n\rightarrow \infty.
\end{eqnarray*}
Similarly, we can prove for $i=2$:
\begin{eqnarray*}
\|A_{n}^{\alpha,\rho}(e_2;x)-e_2\|_2&=&\displaystyle \sup_{x\in [0,\infty)}\frac{|A_{n}^{\alpha,\rho}(e_2;x)-x^2|}{1+x^2}\\
&\leq & \displaystyle \sup_{x\in[0,\infty)}\frac{x^2}{1+x^2}\left|\frac{{\rho}^2n^2+(4\alpha-3){\rho}^2n}{(n\rho-1)(n\rho-2)}\right|\\
&&+\displaystyle \sup_{x\in[0,\infty)}\frac{x}{1+x^2}\left|\frac{{\rho}^2(n+4\alpha-4)+3\rho n-6\rho(1-\alpha)}{(n\rho-1)(n\rho-2)}\right|\\
&&+\displaystyle \sup_{x\in[0,\infty)}\frac{1}{1+x^2}\left|\frac{2}{(n\rho-1)(n\rho-2)}\right|,
\end{eqnarray*}
which gives us\qquad \qquad $\|A_{n}^{\alpha,\rho}(e_2;x)-e_2\|_2=0\,\,\mbox{as}\,\, n\rightarrow \infty.$\\
This completes the proof.
\end{proof}

\subsection{Pointwise Approximation Properties by $A_{n}^{\alpha,\rho}(f;x)$}
\begin{thm}
Let $f \in C_B[0,\infty)$ and $x \in [0,b]$ such that $b>0,$ then we have:
$$| A_{n}^{\alpha,\rho}(f;x)-f(x)|\leq 2\omega\left(f;\sqrt{A_{n}^{\alpha,\rho}((t-x)^2;x)}\,\right).$$
\end{thm}

\begin{proof}
From Definition \ref{def1}, we have:
\begin{eqnarray*}
|A_{n}^{\alpha,\rho}(f;x)-f(x)| &=& \left|\left(\sum_{k=0}^{\infty}p_{n,k}^{\alpha}(x)\int_{0}^{\infty}\mu_{n,k}^{\rho}(t) \,f(t) \,dt\right)-f(x)\right|\\
& \leq & \sum_{k=0}^{\infty}p_{n,k}^{\alpha}(x)\int_{0}^{\infty} \mu_{n,k}^{\rho}(t) \,|f(t)-f(x)| \,dt\\
& \leq & \sum_{k=0}^{\infty}p_{n,k}^{\alpha}(x)\int_0^{\infty}\mu_{n,k}^{\rho}(t)  \,\omega(f;|t-x|) \,dt.
\end{eqnarray*}
By using the property of $\omega(f;\delta \lambda)\leq (1+\lambda)\,\omega(f;\delta)$ with $\lambda=\dfrac{1}{\delta}|t-x|$, we get:
\begin{eqnarray}
|A_{n}^{\alpha,\rho}(f;x)-f(x)|& \leq & \sum_{k=0}^{\infty}p_{n,k}^{\alpha}(x)\int_0^{\infty} \mu_{n,k}^{\rho}(t)\left(1+\frac{1}{\delta}\,|t-x|\right)\omega(f;\delta)\,dt\nonumber\\
&=& \sum_{k=0}^{\infty}p_{n,k}^{\alpha}(x)\int_0^{\infty}\mu_{n,k}^{\rho}(t)\,\omega(f;\delta)\,dt\nonumber\\
&&+\frac{1}{\delta}\sum_{k=0}^{\infty}p_{n,k}^{\alpha}(x)\int_0^{\infty}\mu_{n,k}^{\rho}(t)\,|t-x|\,\omega(f;\delta)\,dt\nonumber\\
&=& \omega(f;\delta)+\frac{1}{\delta}\sum_{k=0}^{\infty}p_{n,k}^{\alpha}(x)\int_0^{\infty}\mu_{n,k}^{\rho}(t)\,|t-x|\,\omega(f;\delta)\,dt\nonumber\\
&=& \left(1+\frac{1}{\delta}\sum_{k=0}^{\infty}p_{n,k}^{\alpha}(x)\int_0^{\infty}\mu_{n,k}^{\rho}(t)\,|t-x|\,dt\right)\omega(f;\delta).
\end{eqnarray}
Now, by using Holder's inequality, we have:
\begin{eqnarray*}
\sum_{k=0}^{\infty}p_{n,k}^{\alpha}(x)\int_0^{\infty}\mu_{n,k}^{\rho}(t)\,|t-x|\,dt & \leq & \sum_{k=0}^{\infty}p_{n,k}^{\alpha}(x)\left[\int_0^{\infty}\mu_{n,k}^{\rho}(t)\,dt\right]^{\frac{1}{2}}.\left[\int_0^{\infty}\mu_{n,k}^{\rho}(t)(t-x)^2\,dt\right]^{\frac{1}{2}}\\
&=& \sum_{k=0}^{\infty}p_{n,k}^{\alpha}(x)\left[\int_0^{\infty}\mu_{n,k}^{\rho}(t)(t-x)^2dt\right]^{\frac{1}{2}} \qquad \qquad \qquad \left[\because \int_0^{\infty}\mu_{n,k}^{\rho}(t)\,dt=1\right]\\
& \leq &\left[\sum_{k=0}^{\infty}p_{n,k}^{\alpha}(x)\right]^{\frac{1}{2}}.\left[\sum_{k=0}^{\infty}p_{n,k}^{\alpha}(x)\int_0^{\infty}\mu_{n,k}^{\rho}(t)(t-x)^2dt\right]^{\frac{1}{2}} \left[\because \sum_{k=0}^{\infty}p_{n,k}^{\alpha}(x)=1\right]\\
&=&\left[\sum_{k=0}^{\infty}p_{n,k}^{\alpha}(x)\int_0^{\infty}\mu_{n,k}^{\rho}(t)(t-x)^2dt\right]^{\frac{1}{2}}= [A_{n}^{\alpha,\rho}((t-x)^2;x)]^{\frac{1}{2}}.
\end{eqnarray*}
Thus (4.1) reduces to:
\begin{eqnarray*}
| A_{n}^{\alpha,\rho}(f;x)-f(x)|& \leq &\left(1+\frac{1}{\delta} \sqrt{A_{n}^{\alpha,\rho}((t-x)^2;x)}\right)\omega(f;\delta).
\end{eqnarray*}
On choosing $\delta=\sqrt{A_{n}^{\alpha,\rho}((t-x)^2;x)},$ we obtain the desired result.
\end{proof}
\vskip0.1in
\begin{thm}
For every $f\in C_B^2[0,\infty)$
\begin{eqnarray*}
\displaystyle\lim_{n \rightarrow \infty}(n\rho-1)\left[ A_{n}^{\alpha,\rho}(f;x)-f(x)\right]&=&[1-2\rho(1-\alpha)x+1]f^{\prime}(x)+(\rho+1)x(x+1)\frac{f^{\prime\prime}(x)}{2}
\end{eqnarray*}
uniformly for $0\leq x\leq b,$ where $b>0.$
\end{thm}
\begin{proof}
The Taylor's polynomial for $f \in C_B^2[0,\infty)$, i.e.
\begin{eqnarray}\label{j1}
f(t)=f(x)+(t-x)f^{\prime}(x)+\frac{(t-x)^2}{2}f^{\prime\prime}(x)+\epsilon(x,a)(t-x)^2\,\quad \mbox{where} \quad a\in [0,\infty)
\end{eqnarray}
having the conditions: $$\epsilon(x,a)\in C_B[0,\infty) \,\mbox{and}\,\displaystyle \lim_{x\rightarrow a}\epsilon(x,a)=0.$$
Now, applying $A_{n}^{\alpha,\rho}(.;.)$ and then multiplying with $(n\rho-1)$ in both sides of (\ref{j1}):
\begin{eqnarray}
(n\rho-1)(A_{n}^{\alpha,\rho}(f;x)-f(x))&=&(n\rho-1)\bigg[A_{n}^{\alpha,\rho}(t-x;x)f^{\prime}(x)+\frac{1}{2}A_{n}^{\alpha,\rho}((t-x)^2;x)f^{\prime\prime}(x)\\
&&+A_{n}^{\alpha,\rho}(\epsilon(x,a)(t-x)^2;x)\bigg]{\nonumber}\\
&=&\bigg[x(1-2\rho(1-\alpha))+1\bigg]f^{\prime}(x)+ \left[(\rho+1)x(x+1)+\frac{2(\rho+1)x(x+1)}{n\rho-2}\right.{\nonumber}\\
&&\left.+\frac{-8\rho(1-\alpha)x^2+2x^2-4{\rho}^2(1-\alpha)x-6\rho(1-\alpha)x+4x+2}{n\rho-2}\right]\frac{f^{\prime\prime}(x)}{2}{\nonumber}\\
&&+(n\rho-1)A_{n}^{\alpha,\rho}(\epsilon(x,a)(t-x)^2;x).\label{eq}
\end{eqnarray}
By Cauchy-Schwartz inequality on the last term of above equation, we get:
\begin{eqnarray*}
(n\rho -1)A_{n}^{\alpha,\rho}(\epsilon(x,a)(t-x)^2;x) &\leq &(n\rho-1) \sqrt{A_{n}^{\alpha,\rho}({\epsilon}^2(x,a);x)}.\sqrt{A_{n}^{\alpha,\rho}((t-x)^4;x)}.
\end{eqnarray*}
Hence, by using Lemma \ref{cenm}, we have:\\
$$\lim_{n\rightarrow \infty}(n\rho-1)A_{n}^{\alpha,\rho}(\epsilon(x,a)(t-x)^2;x)=0.$$
Thus, by taking limit as $n\rightarrow \infty$ on equation (\ref{eq}), we obtain the required result.
\end{proof}
\vskip0.1in
\begin{thm}
Let $f \in Lip_{\mathcal{M}}^{\gamma}$ with $\mathcal{M}>0$ and $0<\gamma \leq 1.$Then the operators $A_{n}^{\alpha,\rho}(.;.)$ satisfy
\begin{eqnarray*}
|A_{n}^{\alpha,\rho}(f;x)-f(x)| \leq \mathcal{M}\left(A_{n}^{\alpha,\rho}((t-x)^2;x)\right)^{\frac{\gamma}{2}}.
\end{eqnarray*}
\end{thm}

\begin{proof}
By using the linearity of the operators $A_{n}^{\alpha,\rho}(f;x)$, we obtain:
\begin{eqnarray*}
|A_{n}^{\alpha,\rho}(f;x)-f(x)|& \leq & A_{n}^{\alpha,\rho}(|f(t)-f(x)|;x).
\end{eqnarray*}
From Definition \ref{def2}, we have
\begin{eqnarray*}
|A_{n}^{\alpha,\rho}(f;x)-f(x)|& \leq & A_{n}^{\alpha,\rho}(\mathcal{M}|t-x|^{\gamma};x)=\mathcal{M} \sum_{k=0}^{\infty}p_{n,k}^{\alpha}(x)\int_0^{\infty}\mu_{n,k}^{\rho}(t)|t-x|^{\gamma}dt\\
&\leq & \mathcal{M}\sum_{k=0}^{\infty}p_{n,k}^{\alpha}(x)\left[\int_{0}^{\infty}\mu_{n,k}^{\rho}(t)dt\right]^{\frac{2-\gamma}{2}}
\left[\int_{0}^{\infty}\mu_{n,k}^{\rho}(t)(t-x)^2dt\right]^{\frac{\gamma}{2}} \\
&&\qquad\qquad~~~~~~~~~~~~~~~~~~~~~~~~~~~~~~~~~~ \qquad \,\,[\mbox{By Holder's inequality with} \,\,p=\frac{2}{2-\gamma}\,\,\mbox{and}\,\,q=\frac{2}{\gamma} ]\\
&= & \mathcal{M}\sum_{k=0}^{\infty}p_{n,k}^{\alpha}(x)\left[\int_{0}^{\infty}\mu_{n,k}^{\rho}(t)(t-x)^2dt\right]^{\frac{\gamma}{2}}\\
&\leq & \mathcal{M}\left[\sum_{k=0}^{\infty}p_{n,k}^{\alpha}(x)\right]^{\frac{2-\gamma}{2}}
\left[\sum_{k=0}^{\infty}p_{n,k}^{\alpha}(x)\int_{0}^{\infty}\mu_{n,k}^{\rho}(t)(t-x)^2dt\right]^{\frac{\gamma}{2}}\\
&=& \mathcal{M}\left(A_{n}^{\alpha,\rho}((t-x)^2;x)\right)^{\frac{\gamma}{2}}.
\end{eqnarray*}
Hence, the proof is completed.
\end{proof}

\begin{thm}{\label{thm7}}
Let $f \in C_B^2[0,\infty)$ and $n\rho >2$, then:
\begin{eqnarray*}
|A_{n}^{\alpha,\rho}(f;x)-f(x)|\leq  \left(\Gamma_{n}(x)+\frac{\Delta_n(x)}{2}\right)\|f\|_{C_B^2[0,\infty)}
\end{eqnarray*}
where $\quad \Gamma_{n}(x)=A_{n}^{\alpha,\rho}((t-x);x)$ and $\quad \Delta_n(x)=A_{n}^{\alpha,\rho}((t-x)^2;x).$

\end{thm}
\vskip0.2in
\begin{proof}
The second order Taylor's polynomial for $f\in C_B^2[0,\infty)$ and $\theta \in (x,t),$ we have:
\begin{eqnarray*}
f(t)=f(x)+(t-x)f^{\prime}(x)+\frac{(t-x)^2}{2}f^{\prime\prime}(\theta).
\end{eqnarray*}
By applying the operators $A_{n}^{\alpha,\rho}(.;x)$ on both sides, we get:
\begin{eqnarray}\label{j2}
A_{n}^{\alpha,\rho}(f;x)&=&f(x)+A_{n}^{\alpha,\rho}(t-x;x)f^{\prime}(x)+\frac{1}{2}A_{n}^{\alpha,\rho}((t-x)^2f^{\prime\prime}(\theta);x)\nonumber\\
|A_{n}^{\alpha,\rho}(f;x)-f(x)|&\leq &|A_{n}^{\alpha,\rho}(t-x;x)|\|f^{\prime}\|_{C_B[0,\infty)}+\frac{1}{2}|A_{n}^{\alpha,\rho}((t-x)^2;x)|\|f^{\prime\prime}\|_{C_B[0,\infty)}.
\end{eqnarray}
Now from equation (\ref{norm}), we obtain the following inequalities,
\begin{eqnarray*}
 \|f^{\prime}\|_{C_B[0,\infty)}\leq \|f\|_{C_B^2[0,\infty)}\,\,\mbox{and}\,\,\,\|f^{\prime\prime}\|_{C_B[0,\infty)} \leq \|f\|_{C_B^2[0,\infty)}.
\end{eqnarray*}
Thus, by using these inequalities in equation (\ref{j2}), we get the required upper bound for the error of approximation.
\end{proof}
\vskip0.3in
\begin{thm}
For every $f\in C_B[0,\infty),$ we have:
 \begin{eqnarray*}
|A_{n}^{\alpha,\rho}(f;x)-f(x)|\leq 2Q\left\{{\omega}_2\left(f;\sqrt{\frac{\Gamma_n(x)}{2}+\frac{\Delta_n(x)}{4}}\right)+min\left(1,\frac{\Gamma_n(x)}{2}+\frac{\Delta_n(x)}{4}\right)\|f\|_{C_B[0,\infty)}\right\},\\
\end{eqnarray*}
where $ Q>0$ is any constant, $\Gamma_n(x)$ and $\Delta_n(x)$ are defined as in Theorem \ref{thm7}.
\end{thm}
\vskip0.2in
\begin{proof}
For $g\in C_B^2[0,\infty)$, we have:

\begin{eqnarray}\label{1}
|A_{n}^{\alpha,\rho}(f;x)-f(x)|&=&|A_{n}^{\alpha,\rho}(f;x)-A_{n}^{\alpha,\rho}(g;x)+A_{n}^{\alpha,\rho}(g;x)-g(x)+g(x)-f(x)|\nonumber\\
&\leq &|A_{n}^{\alpha,\rho}(f-g;x)|+|g(x)-f(x)|+|A_{n}^{\alpha,\rho}(g;x)-g(x)|.\\ \nonumber
\end{eqnarray}
Since
\begin{eqnarray*}
A_{n}^{\alpha,\rho}(f;x)\rightarrow f(x),\,\,\qquad \qquad A_{n}^{\alpha,\rho}(g;x)\rightarrow g(x),
\end{eqnarray*}
then
\begin{eqnarray}\label{2}
|A_{n}^{\alpha,\rho}(f-g;x)|\rightarrow |f(x)-g(x)|\leq\|f-g\|_{C_B[0,\infty)}.\\\nonumber
\end{eqnarray}
By using equations (\ref{1}), (\ref{2}) and Theorem \ref{thm7}, we get:
\begin{eqnarray*}
|A_{n}^{\alpha,\rho}(f;x)-f(x)|&\leq &2\|f-g\|_{C_B[0,\infty)}+\left(\Gamma_{n}(x)+\frac{\Delta_n(x)}{2}\right)\|g\|_{C_B[0,\infty)}\\
&=&2\bigg(\|f-g\|_{C_B[0,\infty)}+\left(\frac{\Gamma_{n}(x)}{2}+\frac{\Delta_n(x)}{4}\right)\|g\|_{C_B[0,\infty)}\bigg).
\end{eqnarray*}
Now taking the infimum over all $g\in C_B^2[0,\infty)$ and using the Definition \ref{def3}, we have:
\begin{eqnarray*}
|A_{n}^{\alpha,\rho}(f;x)-f(x)|\leq 2K_2\left(f;\frac{\Gamma_{n}(x)}{2}+\frac{\Delta_n(x)}{4}\right).
\end{eqnarray*}
Also, from \cite{11}, using the equivalence relation of second order $K-$functional and second order modulus of continuity, i.e.
$$K_2(f;\delta)\leq Q \big({\omega}_2(f;\sqrt{\delta})+min(1,\delta)\|f\|_{C_B[0,\infty)}\big), \quad Q > 0$$
Thus:
\begin{eqnarray*}
|A_{n}^{\alpha,\rho}(f;x)-f(x)|\leq 2Q\left\{{\omega}_2\left(f;\sqrt{\frac{\Gamma_{n}(x)}{2}+\frac{\Delta_n(x)}{4}}\right)+min\left(1,\frac{\Gamma_{n}(x)}{2}+\frac{\Delta_n(x)}{4}\right)\|f\|_{C_B[0,\infty)}\right\}.
\end{eqnarray*}
Hence, we obtain the result.
\end{proof}
\vskip0.3in

\section{Numerical Verifications}
\vskip0.2in
In this section, we give the numerical examples to verify the theoretical results that we have proved in the previous section. Here, we compare the Durrmeyer variant of $\alpha-$Baskakov operators with our operators $A_{n}^{\alpha,\rho}(f;x)$ for different values of $\rho.$
\vskip0.3in
\begin{example}
Let $f(x)=\sqrt x$. In this example, we present the convergence of our operators $A_{n}^{\alpha,\rho}(f,x)$ for $n=20, \alpha=0.1$ and certain values of $\rho$ i.e. $\rho=1, 5, 0.5.$
\vskip0.1in
Also, we have given the error of approximation $E_{n}^{\alpha,\rho}(f,x)=|f(x)-A_{n}^{\alpha,\rho}(f;x)|$ of our operators $A_{n}^{\alpha,\rho}(f;x)$ from the function $f(x)$ for $n=20, \alpha=0.1$ and $\rho=1, 5, 0.5.$ \\
\vskip0.1in
From the given figures, we can easily see that the operators $A_{n}^{\alpha,\rho}(f;x)$ give the better approximation at $\rho=0.5.$
\newpage
\begin{center}
{Figure 1. Approximation process for $\alpha=0.1$}\\
\includegraphics[width=0.75\columnwidth]{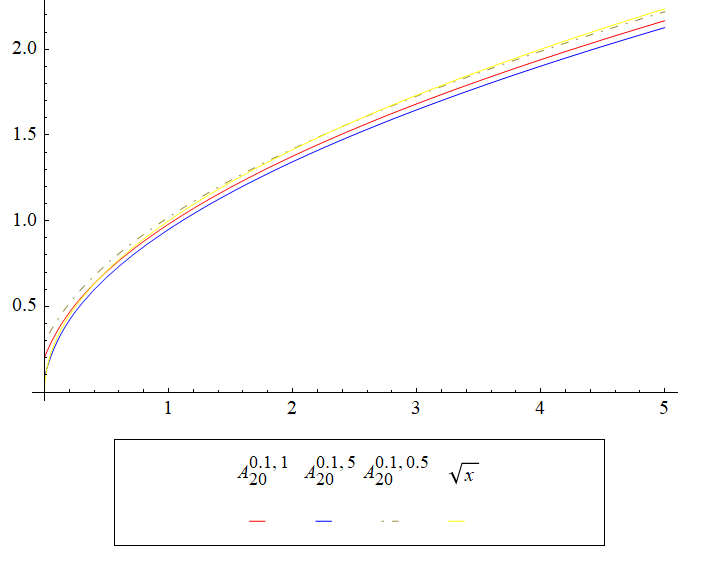}
\end{center}
\vspace{0.3in}
\begin{center}
{Figure 2. Error estimation}\\
\includegraphics[width=0.75\columnwidth]{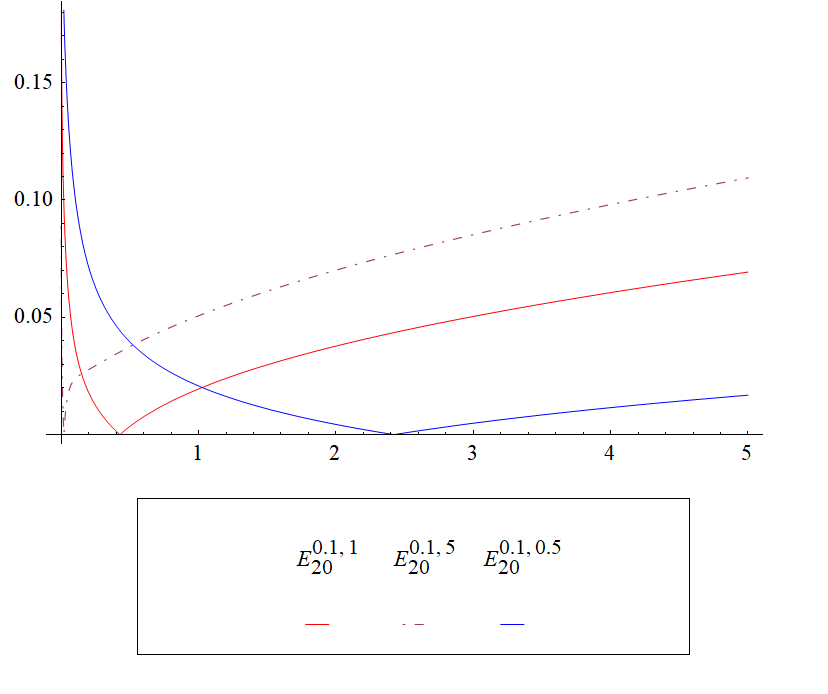}
\end{center}
\newpage
Also, we give the convergence and the error approximation at $\alpha=1$ with $n=20$ and $\rho=1,5,0.5,$ which is shown in figure 3 and figure 4. For these values, it is clear that approximation is better at $\rho=5.$\\
\vskip0.3in
\begin{center}
{Figure 3. Approximation process for $\alpha=1$}\\
\includegraphics[width=0.65\columnwidth]{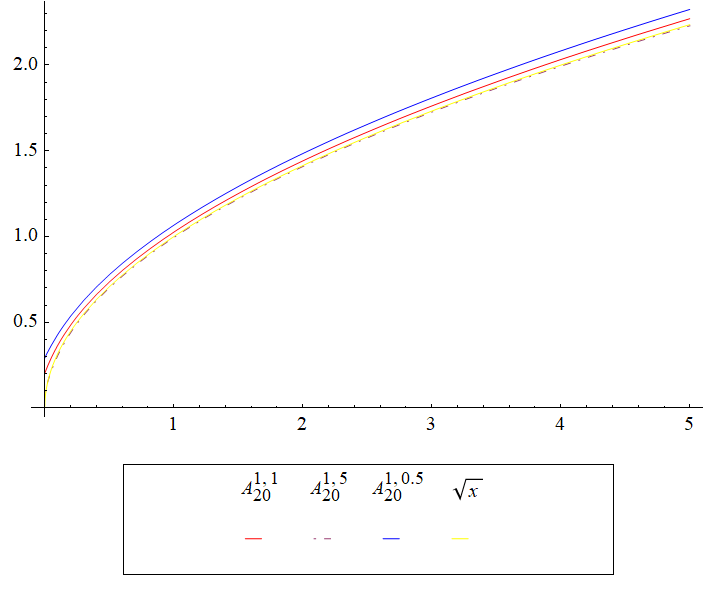}
\end{center}
\vspace{0.3in}
\begin{center}
{Figure 4. Error estimation}\\
\includegraphics[width=0.65\columnwidth]{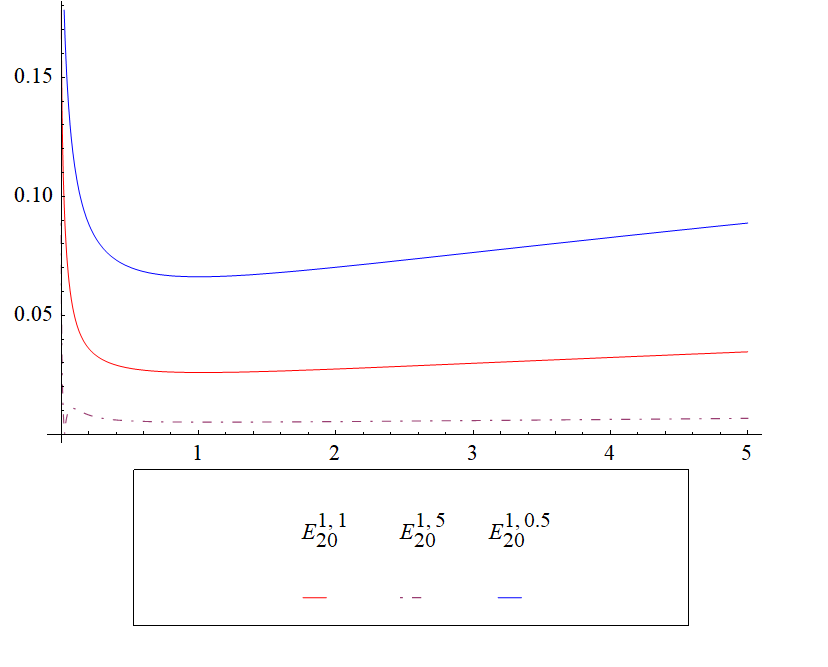}
\end{center}
\end{example}
\newpage
\begin{example}
Let $f(x)=x^2+5x+2$. Here, we give the convergence of our operators $A_{n}^{\alpha,\rho}(f;x)$ at $n=20,\alpha=0.7$ and some values of $\rho=1, 5, 0.3.$
Also, we give the error of approximation $E_{n}^{\alpha,\rho}(f,x)$ of our operators $A_{n}^{\alpha,\rho}(f;x)$ from the function $f(x)$ at the chosen values. 
\vskip0.3in
\begin{center}
{Figure 5. Approximation process for $\alpha=0.7$}\\
\includegraphics[width=0.63\columnwidth]{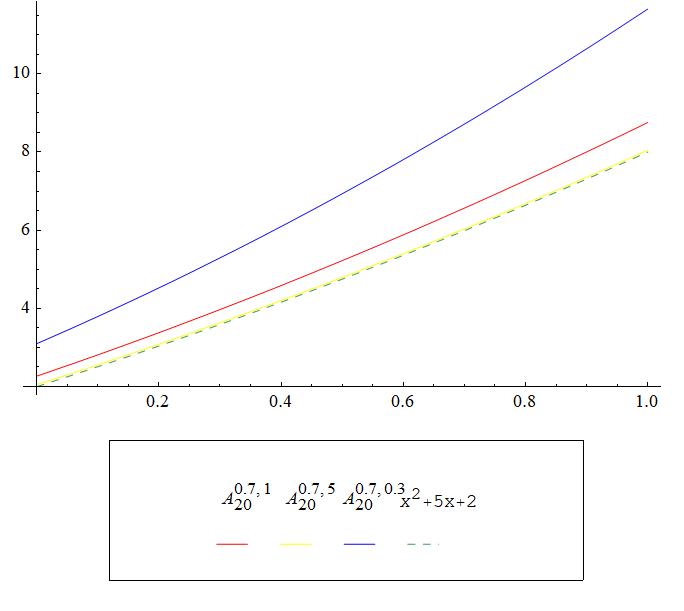}
\end{center}
\vskip0.3in
\begin{center}
{Figure 6. Error estimation}\\
\includegraphics[width=0.63\columnwidth]{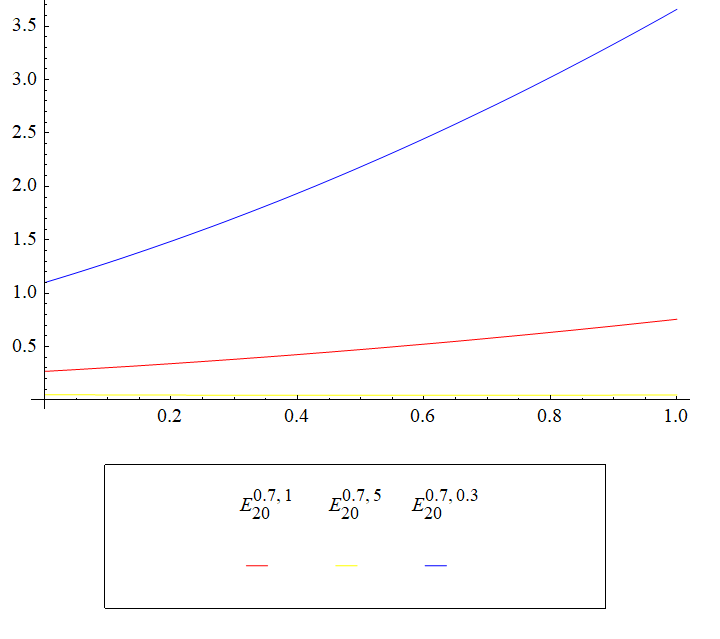}
\end{center}
From the above figures we can easily see that the operators $A_{n}^{\alpha,\rho}(f;x)$ give the better approximation at $\rho=5.$
\end{example}
\vskip0.3in
\textbf{Conclusion}: The operators $A_{n}^{\alpha,\rho}(f;x)$ converge uniformly to continuous bounded function $f(x).$ By these operators, we have verified that generalization of the positive linear operators can provide the better approximation than the classical ones, which is shown graphically through the examples. Here, we have seen that the convergence of the operators $A_{n}^{\alpha,\rho}(f;x)$ is independent of parameters $\alpha$ and $\rho$, but the error of approximation depends on these parameters.



\begin{thebibliography}{99}
\bibitem{Baka}P. N. Agrawal and M. Goyal, Generalized Baskakov Kantorovich operators, Filomat, 31 (19) (2019), 6131-6151.
\bibitem{albas}A. Aral and H. Erbay, Parametric generalization of Baskakov operators, Math. Commun. 24 (2019), 119-131.
\bibitem{qbas}A. Aral and V. Gupta, On $q-$Baskakov type operators, Demonstratio Math. 42 (2009), 109-122.
\bibitem{genqbas} A. Aral and V. Gupta, Generalized $q-$Baskakov operators, Math. Slovaca 61 (2011), 619-634.
\bibitem{BAS}V. Baskakov, An instance of a sequence of linear positive operators in the space of
continuous functions, Doklady Akademii Nauk SSSR 113 (1957), 249-251.
\bibitem{BER} S. N. Bernstein, D$\acute{e}$monstration du th$\acute{e}$or$\acute{e}$me de Weierstrass fond$\acute{e}$e sur le calculdes probabiliti$\acute{e}$s. Commun. Soc. Math. Kharkov 13 (1913), 1-2.
\bibitem{PLB} P. L. Butzer, Linear combinations of Bernstein polynomials. Canad. J. Math. 5 (2) (1953), 559-567.
\bibitem{lambe}Q. Cai, B. Lian and G. Zhou, Approximation properties of $\lambda$-Bernstein, J. Inequal. Appl. 61 (2018), 1-11.
\bibitem{CHE} X. Chen, J. Tan, Z. Liu and J. Xie, Approximation of functions by a new family of generalized Bernstein operators, J. Math. Anal. Appl. 450 (2017), 244-261.
\bibitem{11} A. Ciupa, A class of integral Favard-Sz$\acute{a}$sz type operators, Stud. Univ. Babe\c{s}-Bolyai, Math. 40 (1) (1995), 39-47.
\bibitem{DT} Z. Ditzian and V. Totik, Moduli of Smoothness. Springer, New York (1987).
\bibitem{two} A. D. Gadjiev and A. M. Ghorbanalizadeh, Approximation properties of a new type Bernstein-Stancu polynomials of one and two variables, Appl. Math. Comput. 216 (3) (2010), 890-901.
\bibitem{kanb} M. Goyal and P. N. Agrawal, B\`{e}zier variant of generalized Baskakov-Kantorovich operators, Boll. Dell'Unione. Mat. Ital. 8 (4) (2016), 229-238.
\bibitem{Bask} V. Gupta, A note on modified Baskakov type operators, Approx. Theory Appl. 10 (1994), 74-78.
\bibitem{QBER} N. I. Mahmudov and P. Sabancygil, Some approximation properties of Lupas $q-$analogue of Bernstein operators, arXiv:1012.4245v1 [math.FA] 20 Dec 2010.
\bibitem{CAM} C. A. Micchelli, Saturation classes and iterates of operators. Ph. D. Thesis, Stanford University (1969).
\bibitem{geba} V. Mihesan, Uniform approximation with positive linear operators generated by generalized Baskakov method, Automat. Comput. Appl. Math. 7 (1998), 34-37.
\bibitem{dalba} M. Nasiruzzaman, N. Rao, S. Wazir and R. Kumar, Approximation on parametric extension of Baskakov-Durrmeyer operators on weighted spaces,J. Inequal. Appl. 103 (2019), 1-11.
\bibitem{kfu}J. Peetre, A Theory of Interpolation of Normed Spaces. Noteas de Mathematica, Instituto de Mathem$\acute{a}$tica Pura e Applicada, Conselho Nacional de Pesquidas, Rio de Janeiro, 39 (1968).
\bibitem{Ba} S. Pethe, On the Baskakov operator, Indian J. Math. 26 (1984), 43-48.
\bibitem{stan} D. D. Stancu, Approximation of functions by a new class of linear polynomial operators, Rev. Roumaine Math. Pure Appl., 13 (1968), 1173-1194.
\bibitem{szs}O. Sz$\acute{a}$sz, Generalization of S. Bernstein's polynomial to the infinite interval. J. Res. Nat. Bur. Standards 45 (1950), 239-245.
\bibitem{wafi}A. Wafi and S. Khatoon, On the order of approximation of functions by generalized Baskakov operators, Indian J. Pure Appl. Math. 35 (2004), 347-358.
\bibitem{bas} C. Zhang and Z. Zhu, Preservation properties of the Baskakov-Kantrovich operators, Comput. Math. Appl. 57 (2009), 1450-1455.
\end{thebibliography}
\end{document}